\documentclass[a4paper,10pt]{article}
\usepackage{cmap}
\usepackage[T2A]{fontenc}
\usepackage[utf8]{inputenc} % [cp866][cp1251]
\usepackage[english]{babel}
\usepackage{amsmath,amssymb}
\usepackage{graphicx}
\usepackage{ifthen}
\usepackage{multirow}
\makeatletter
\def\@seccntformat#1{\csname the#1\endcsname.\ } % the column after a section number
\def\@biblabel#1{#1.} % number style in the bibliography list
\makeatother

\makeatletter
\newcounter{language}
\@addtoreset{section}{language}
\@addtoreset{equation}{language}
\@addtoreset{footnote}{language}
\@addtoreset{Hfootnote}{language} % the package hyperref uses this counter for footnotes
\makeatother

\date{}

\newenvironment{proof}[1][\hspace{-1.0ex}]%
{\par\addvspace{1mm}{\sc Proof\hspace{1.0ex}{#1}.} }%
{\quad$\blacktriangle$\par\addvspace{1mm}}

\newif\ifNoRemark
\def\addtheorem#1#2#3#4{
\ifthenelse{\equal{#2}{}}{}%
{\ifthenelse{\expandafter\isundefined\csname the#2\endcsname}{\newcounter{#2}}{}}
\newenvironment{#1}[1][\global\NoRemarktrue]% No Remark by default
{\par\addvspace{2mm plus 0.5mm minus 0.2mm}\noindent % new paragraph without indent
{\bf #3}\ifthenelse{\equal{#2}{}}{}%
{\refstepcounter{#2}{\bf ~\csname the#2\endcsname}}%
{\bf \vphantom{##1}\ifNoRemark.\ \else\ (##1).\fi}\begingroup #4}%
{\endgroup\par\addvspace{1mm plus 0.5mm minus 0.2mm}\global\NoRemarkfalse}
\expandafter\newcommand\csname b#1\endcsname{\begin{#1}}
\expandafter\newcommand\csname e#1\endcsname{\end{#1}}
}

\addtheorem{theorem}{thrm}{Theorem}{\sl}
\addtheorem{lemma}{lmm}{Lemma}{\sl}
\addtheorem{corollary}{crll}{Corollary}{\sl}
\addtheorem{proposition}{prpstn}{Proposition}{\sl}
\addtheorem{example}{xmpl}{Example}{}

\makeatletter
\@addtoreset{thrm}{language}
\@addtoreset{lmm}{language}
\@addtoreset{prpstn}{language}
\@addtoreset{problem}{language}
\makeatother

\title{On the number of resolvable Steiner triple systems \\ of small $3$-rank%
\thanks{This research is supported by
National Natural Science Foundation of China (61672036),
Excellent Youth Foundation of Natural Science Foundation of Anhui Province (1808085J20) and the
Program of fundamental scientific researches of the Siberian Branch of the Russian Academy of Sciences
No.I.5.1. (proj. No.0314-2019-0016).}
}

\author{
M. Shi, L. Xu%
\thanks{M. Shi and L. Xu are with the School of Mathematical Sciences, Anhui University, Hefei, 230601, China}%
,
D. S. Krotov%
\thanks{D. Krotov is with the Sobolev Institute of Mathematics, Novosibirsk, 630090, Russia}
}

\newcommand\aA{\mathcal A}
\newcommand\bB{\mathcal B}
\newcommand\cC{\mathcal C}

\newcommand\gG{\mathcal G}
\newcommand\tT{\mathcal T}
\newcommand\pP{\mathcal P}

\newcommand\STS{{\rm STS}}
\newcommand\TD{{\rm TD}}
\newcommand\AG{{\rm AG}}

\usepackage{color}

\begin{document}
\thispagestyle{empty}
\maketitle

\begin{abstract}
In a recent work,
Jungnickel, Magliveras, Tonchev, and Wassermann
derived an overexponential lower bound on
 the number of nonisomorphic resolvable Steiner triple systems (STS)
 of order $v$, where $v=3^k$, and $3$-rank $v-k$.
 We develop an approach to generalize this bound
 and estimate the number of isomorphism classes
 of STS$(v)$ of rank $v-k-1$ for an arbitrary $v$ of form $3^kT$.
\end{abstract}

\section{Introduction}

A \emph{Steiner triple system} of order $v$, or \emph{\STS$(v)$},
is a pair $(S,\bB)$ from a finite set $S$
(called the \emph{support},
or the \emph{point set}, of the \STS)
of cardinality $v$ and a collection $\bB$ of $3$-subsets of $S$,
called \emph{blocks}, such that
every two distinct elements of $S$ meet in exactly one block.
A transversal design \TD$(k,w)$
% (in this paper, we only consider the case $k=3$)
is a triple $(S, \gG, \bB)$ that consists of a point set $S$ of cardinality $kw$,
a partition $\gG$ of $S$ into $k$ subsets, \emph{groups}, of cardinality $w$,
and a collection $\bB$ of $k$-subsets of $S$, \emph{blocks},
such that every block intersects every group in exactly one point and every two points in different groups
meet in exactly one block.
It is convenient to identify the system, STS or TD, with its block set
(indeed, the support and, in the case of \TD, the groups are uniquely determined by the block set).
% With this agreement, it is correct to say that an STS $\bB$ can include, as a subset,
% some STS or TD $\cC$, in which case $\cC$ is called
% a \emph{sub-STS} or \emph{sub-TD} of $\bB$, respectively.
Two systems, STS or TD, are called \emph{isomorphic}
if there is a bijection between their supports,
an \emph{isomorphism},
that sends the blocks of one system to the blocks of the other.
An isomorphism of a system $\bB$ to itself is called an \emph{automorphism};
the set of all automorphisms
of $\bB$ is denoted by $\mathrm{Aut}(\bB)$.

Given a system, STS or TD, with the support $S$, a \emph{parallel class} is a partition of $S$
into blocks. A partition of the system into parallel classes
is called a \emph{resolution} or \emph{parallelism}.
Note that a resolution of an STS$(v)$ consists of $(v-1)/2$ parallel classes,
while a resolution of an TD$(3,w)$ has $w$ parallel classes.
A system that admits at least one resolution is called \emph{resolvable}.
The problem of existence of resolvable STS of order $6n+3$
was suggested by Kirkman in
\cite{Kirkman:1847} and completely solved by
Ray-Chaudhuri and Wilson in 1971 \cite{RChWil:kirkman} (as mentioned in \cite[p.13]{HCD}, the problem
``had in fact already been solved by the Chinese mathematician Lu Jiaxi at least eight
years previously, but the solution had remained unpublished because of the political
upheavals of the time'').
On the other hand, the existence of infinite families of non-resolvable Steiner triple systems
of order divisible by $3$ was established in~\cite{LiRees:2005}.

For any prime $p$, the $p$-rank of a Steiner triple system $(\{0,\ldots,v-1\},\bB)$
is the dimension of the linear span over GF$(p)$ of the set of characteristic $\{0,1\}$-vectors
$(b_0,b_1,\ldots,b_{v-1})$, $b_i=1 \Leftrightarrow i\in B$ of the blocks $B\in \bB$.
As shown in \cite{DHV:1978}, the $p$-rank of every STS$(v)$ is $v$
for all prime $p$ except $2$ and $3$.
A sequence of papers
is devoted to the enumeration of Steiner triple systems in accordance
with the $2$- and $3$-rank, see a survey in \cite{SXK:2019}.
%, and also latest works \cite{JunTon:18+,JMTW:STS27}.
In \cite{JMTW:STS27}, Jungnickel et al. obtained a computer-aided classification
of STS$(27)$ of the next-to-minimum $3$-rank $24$. The authors observed that all such systems
are resolvable and proved this property theoretically for any STS$(3^k)$ of $3$-rank $3^k-k$.
As a consequence, they derived a lower bound on the number of nonisomorphic resolvable STS$(3^k)$.
The goal of the current research is to generalize this bound to an arbitrary order of form $3^kT$,
$T$ odd. More specifically, we evaluate the number of resolvable STS$(3^kT)$ of rank $3^kT-k-1$,
$k\to\infty$,
for every odd $T$ except the case when $T\equiv 1 \bmod 6$ is not a prime power
(which could be resolved if a resolvable STS$(3T)$ of rank $3^kT-2$ is constructed).

The next section contains some preliminary results.
In Section~\ref{s:min}, we consider the case when $T\equiv 15 \bmod 18$ (Theorem~\ref{th:1}) or
$T\equiv 1 \bmod 6$ (Theorem~\ref{th:2}); that is, when $v-k-1$ is the minimum possible $3$-rank
of STS$(v)$, $v=3^k T$. The remaining case  $T\equiv 3,9 \bmod 18$ are solved in Section~\ref{s:nonmin},
Theorem~\ref{th:1'}.
In the conclusion section, we consider some further research topics.

\section{Preliminary results}

Denote by $G(v,k)$ the linear span, over GF$(3)$, of the all-one vector
and the rows of the $k\times v$ matrix that consists of all possible columns ordered lexicographically,
each column occurring $v/3^k$ times.

Denote by \AG$(k)$ the Steiner triple system $(S,\aA)$ with the point set $S=\{0,\ldots,3^k-1\}$
and the block set $\aA = \{\{a,b,c\} | t(a)+t(b)+t(c)=0\}$, where $t(0)$, \ldots, $t(3^k-1)$ are
the ternary
$k$-tuples ordered lexicographically. \AG$(k)$ is known as the \emph{affine geometry} (over GF$(3)$); trivially, it has a resolution,
where for every block $\{a,b,c\}$, all blocks $\{a',b',c'\}$ such that $t(a')-t(a)=t(b')-t(b)=t(c')-t(c)$
belong to the same parallel class. We call it the \emph{standard resolution} of \AG$(k)$.

\begin{proposition}[\cite{JunTon:18+}]\label{p:any}
 A Steiner triple system of order $v$ has $3$-rank at most $v-k-1$
 if and only if it is isomorphic to an STS orthogonal to $G(v,k)$.
\end{proposition}
\begin{proposition}[\cite{JunTon:18+}]\label{p:standard}
Let $v = 3^k T$.
Denote $S_i=\{iT,iT+1,\ldots,(i+1)T-1\}$, $i=0,\ldots,3^k-1$.
 An STS $(S,\bB)$ is orthogonal to  $G(v,k)$ if and only if
\begin{equation}\label{eq:3}
 \bB=\bigcup_{i=0}^{3^k-1} \bB_{i} \cup \bigcup_{B\in \AG(k)} \tT_B,
\end{equation}
where $(S_i,\bB_i)$ is an STS, $i=0,\ldots,3^k-1$,
and $(\{S_{i_1},S_{i_2},S_{i_3}\}, \tT_{\{i_1,i_2,i_3\}})$ is a TD,
$\{i_1,i_2,i_3\} \in \AG(k)$.
\end{proposition}

\begin{corollary}\label{c:min}
 Let $v=3^t T'$.
 If $T'\equiv 1 \bmod 6$, then the minimum possible $3$-rank of \STS$(v)$ is $v-t-1$.
 If $T'\equiv 5 \bmod 6$, then the minimum possible $3$-rank of \STS$(v)$ is $v-t$.
\end{corollary}
\begin{proof}
 By Proposition~\ref{p:standard},
 a Steiner triple system of order $v=3^k T$ has $3$-rank at most $v-k-1$
 exists if and only if an \STS$(T)$ exists.
 If $T'\equiv 1 \bmod 6$, then this is true with $T=T'$ and $k=t$.
 If $T'\equiv 5 \bmod 6$, then this is true with $T=3T'$ and $k=t-1$,
 but false for $T=T'$.
\end{proof}

The following lemma was essentially proved in \cite{LuShi};
 we need to formulate the statement in more general form.
\begin{lemma}\label{l:resol}
Assume that the set of points $S$ is divided into $n$ groups $S_0$, \ldots, $S_{n-1}$ of size $m$,
and we are given
\begin{itemize}
 \item an STS $(\{S_0,\ldots,S_{n-1}\},\aA)$ with the groups as the points;
 \item for every group $S_i$, $i=0,\ldots,n-1$, an STS $(S_i,\bB_i)$ with the support $S_i$;
 \item for every block $B=\{S_i,S_j,S_k\}$ from $\aA$,
       a transversal design $\tT_B$ on the groups $S_i$, $S_j$, $S_k$.
\end{itemize}
If all these systems are resolvable, then the union
\begin{equation}\label{eq:_1}
 \bigcup_{i=0}^{n-1} \bB_{i} \cup \bigcup_{B\in \aA} \tT_B
\end{equation}
is a resolvable \STS$(mn)$.
\end{lemma}
\begin{proof}
 The proof is straightforward.
 At first, check the STS property. If two points belong to the same $S_i$,
 then they lie in a unique block from the STS $\bB_i$.
 If they are from different $S_i$, $S_j$,
 then there is a unique $k$ such that $\{S_i,S_j,S_k\} \in \aA$ and,
 by the definition of a transversal design, a unique block
 from $\tT_{\{S_i,S_j,S_k\}}$ that contains these two points.
 So, \eqref{eq:_1} in an STS by the definition.

 Next,
 if $\{\pP_i^j\}_{j=1}^{(m-1)/2}$ is a resolution of $\bB_i$, $i=0,\ldots,n-1$,
 then
 $\{\bigcup_{i=0}^{n-1} \pP_i^j\}_{j=1}^{(m-1)/2}$
 is a partition of the first part of \eqref{eq:_1} into parallel classes.

 If we have a parallel class $\pP$ of $\aA$, and
 for every $B$ in $\pP$, $\{\pP_B^j\}_{j=1}^{m}$ is a resolution of $\tT_B$,
 then $\{\bigcup_{B\in \pP} \pP_B^j\}_{j=1}^{m}$ is a partition of
 $$ \bigcup_{B\in \pP} \tT_B $$
 into parallel classes.
 Unifying over all parallel classes $ \pP $ of a resolution of $\aA$,
 we obtain a partition of the second part of \eqref{eq:_1} into parallel classes.
 So, we have a partition of  \eqref{eq:_1} into parallel classes, and it is a resolvable STS by the definition.
\end{proof}

\section{Resolvable STS of the minimum $3$-rank}\label{s:min}

In this section,
we are interested in the number
of resolvable STS of the minimum $3$-rank,
for a given order $v$.
According to Corollary~\ref{c:min},
there are two cases for the order $v=3^tT'$ of an STS,
$T'\equiv 1 \bmod 6$
and
$T'\equiv 5 \bmod 6$,
which will be considered separately.
Note that in the second case,
% $v=3^tT'$,
% where $T'\equiv 5 \bmod 6$,
the value of $t$ cannot be smaller than $1$;
so, in that case $v=3^{t-1}T$, where $T\equiv 15 \bmod 18$.
We will start with this case,
as it is solved by simpler arguments.

\subsection{$v=3^kT$, where $T\equiv 15 \bmod 18$}\label{s:15}
\begin{theorem}\label{th:1}
 Assume that $v=3^k T$,
 where $T\equiv 15 \bmod 18$.
 The number $N(v,k)$ of isomorphism classes of resolvable
 Steiner triple systems on $v$ points
 with $3$-rank exactly $v-k-1$ satisfies
 \begin{equation}\label{eq:b1}
   N(v,k)
 \geq
 \frac
 {{\widetilde{{N_1}}(T)}^M\cdot{\widetilde{{N_3}}(T)}^{M(M-1)/6}}
 {T!^M\cdot|{\mathrm{AGL}(k,3)}|}
 \end{equation}
where $M=3^{k}$,
$\widetilde{N_1}(T)$ is the number of resolvable \STS$(T)$,
$\widetilde{N_3}(T)$ is the number of resolvable \TD$(3,T)$
(Latin squares of order $T$ having an orthogonal mate),
and $\mathrm{AGL}(k,3)$ is the automorphism group of the affine geometry $\AG(k)$,
$|{\mathrm{AGL}(k,3)}| = 3^k\cdot\prod_{i=0}^{k-1}(3^k-3^i)$.
\end{theorem}
\begin{proof}
We note that $T \equiv 3 \bmod 6$;
so, there exists a resolvable STS$(T)$ \cite{RChWil:kirkman}.
As $T/3 \equiv 5 \bmod {6}$, there is no STS$(T/3)$.
So, $v-k-1$ is the minimum possible rank. We consider the STS's orthogonal to $G(v,k)$.
By Proposition~\ref{p:standard},
there are
exactly ${{{{N_1}}(T)}^M\cdot{{{N_3}}(T)}^{M(M-1)/6}} $ such STS,
where $N_1(T)$ is the number of \STS$(T)$,
$N_3(T)$ is the number of Latin squares of order $T$.
As follows from Lemma~\ref{l:resol},
at least ${{\widetilde{{N_1}}(T)}^M\cdot{\widetilde{{N_3}}(T)}^{M(M-1)/6}}$
of them are resolvable.

It remains to show that at most $ {T!^M\cdot|AGL(k,3)|} $ of these systems can belong to the same
isomorphism class. Indeed, because the $3$-rank of the considered systems is exactly $v-k-1$,
$G(v,k)$ is the dual space of each of them, which is uniquely defined by a system
(this is a crucial observation). Therefore, an isomorphism between two systems
is necessarily an automorphism of $G(v,k)$.
The number of such automorphisms is $ {T!^M\cdot|\mathrm{AGL}(k,3)|} $, see
\cite[Theorem~3.5]{JunTon:18+}.
\end{proof}

\subsection{$v=3^kT$, where $T\equiv 1 \bmod 6$}

In this case,
we cannot use the arguments of Theorem~\ref{th:1}
directly, because $T$ is not divisible by $3$ and hence
there are no resolvable STS$(T)$, $\widetilde N_1(T)=0$.
To solve this,
we will introduce a modified version of Proposition~\ref{p:standard},
where the union~\eqref{eq:3} is rearranged to fit our needs.
We fix $t$ from $0 $ to $k$ and split the set $\{0, \ldots ,3^k-1\}$ into the $3^{k-t}$ groups
$$L_0:=\{0, \ldots ,3^t-1\},\ \ldots,\ L_{3^{k-t}-1} := \{(3^{k-t}-1)3^t, \ldots ,3^k-1\}$$
of size $3^t$.
Denote $$S'_j:=\bigcup_{i\in L_j}S_i, \qquad j=0, \ldots , 3^{k-t}-1.$$
Now, it is easy to understand that the blocks of \AG$(k)$ restricted by $S'_j$
form a sub-STS, actually \AG$(t)$; we denote it by $\AG_j(t)$, $j=0, \ldots , 3^{k-t}-1$.
The set of blocks of \AG$(k)$ that do not belong to any $\AG_j(t)$ is denoted by
$\AG^{(t)}(k)$
(essentially, it is a TD$(3^{k-t},3^{t})$). With this notation, \eqref{eq:3} turns to
\begin{eqnarray}
  \bigcup_{j=0}^{3^{k-t}-1}\bigcup_{i\in L_j} \bB_{i}
 \ \cup \
  \bigcup_{j=0}^{3^{k-t}-1}\bigcup_{B\in \AG_j(t)} \tT_B
  \ \cup \
   \bigcup_{B\in \AG^{(t)}(k)} \tT_B
   \\
   =
     \bigcup_{j=0}^{3^{k-t}-1}
\cC_j
     \ \cup \
   \bigcup_{B\in \AG^{(t)}(k)} \tT_B
   \\ \label{eq:Cj}
   \mbox{where }
   \cC_j := \bigcup_{i\in L_j} \bB_{i}
 \cup
 \bigcup_{B\in \AG_j(t)} \tT_B.
\end{eqnarray}
We now see that $(S'_j, \cC_j)$ is an STS, and it is orthogonal to $G(3^t T,t)$.
Moreover, by Proposition~\ref{p:standard}, any  $({S'_j}, \cC_j)$ orthogonal to $G(3^t T,t)$
can be represented as \eqref{eq:Cj}.
So, we have derived the following generalization of Proposition~\ref{p:standard}.

\begin{theorem}\label{th:standard+}
Let $v = 3^k \cdot T = 3^{k-t}\cdot 3^t\cdot T $.
Denote $S=\{0,\ldots,v-1\}$, $S_i=\{iT,\ldots,(i+1)T-1\}$, $i=0,\ldots,3^k-1$,
and {$S'_j=\{j3^tT,\ldots,(j+1)3^tT-1\}$,} $j=0,\ldots,3^{k-t}-1$.
 An STS $(S,\bB)$ is orthogonal to  $G(v,k)$ if and only if
\begin{equation}\label{eq:4}
 \bB=\bigcup_{j=0}^{3^{k-t}-1} \cC_{j} \cup \bigcup_{B\in \AG^{(t)}(k)} \tT_B,
\end{equation}
where $(S'_j,\cC_j)$ is an STS  orthogonal to {$G(3^t T,t)$},
$j=0,\ldots,3^{k-t}-1$,
and $(\{S_{i_1},S_{i_2},S_{i_3}\}, \tT_{\{i_1,i_2,i_3\}})$ is a TD,
$\{i_1,i_2,i_3\} \in \AG^{(t)}(k)$.
\end{theorem}

\begin{theorem}\label{th:2}
 Assume that $v=3^k T$,
 where $T\equiv 1 \bmod 6$.
 The number $N(v,k)$ of isomorphism classes of resolvable
 Steiner triple systems on $v$ points
 with $3$-rank exactly $v-k-1$ satisfies
 \begin{equation}\label{eq:b2}
   N(v,k)
 \geq
 \frac
 {{\widehat{{N_1}}(3T)}^{M/3}\cdot{\widetilde{{N_3}}(T)}^{M(M-3)/6}}
 {T!^M\cdot|AGL(k,3)|},
 \end{equation}
where $M=3^{k}$;
$\widehat{N_1}({3}T)$ is the number of resolvable \STS$({3}T)$ orthogonal to $G(3T,1)$;
$\widetilde{N_3}(T)$ is the number of Latin squares of order $T$ having an orthogonal mate.
\end{theorem}
\begin{proof}
We first observe that
$$  |\AG^{(1)}(k)| = |\AG(k)|-\frac M3 = \frac{M(M-1)}6- \frac M3 = \frac{M(M-3)}6. $$
So, we can construct
$ {{\widehat{{N_1}}(3T)}^{M/3}\cdot{\widetilde{{N_3}}(T)}^{M(M-3)/6}} $
different systems in the form \eqref{eq:4},
where $t=1$ and all $\cC_j$ and $\tT_B$
are resolvable and $\cC_j$ is orthogonal to $G(3T,1)$.
By Theorem~\ref{th:standard+},
the resulting systems have the rank at most $v-k-1$,
which is the minimum possible value.
Similarly to Lemma~\ref{l:resol},
we find that the resulting systems are resolvable
(note that $\AG^{(1)}(k)$ is resolvable as it is obtained from $\AG(k)$
by deleting the blocks of a standard parallel class).

The proof that representatives of one isomorphism class occur not more than ${T!^M\cdot|AGL(k,3)|}$
times is similar to Theorem~\ref{th:1}.
\end{proof}

To claim that the theorem above gives a nonzero bound,
we need to ensure that $\widehat{{N_1}}(3T)>0$.
For some values of $T$, this fact is guaranteed by one of
Ray-Chaudhuri--Wilson constructions \cite{RChWil:kirkman},
as reflected in the following proposition.
Proving $\widehat{{N_1}}(3T)>0$ for the other values of $T$ remains an open problem.
\begin{proposition}\label{p:RChWil-constr}
For every prime power $T$ such that $T\equiv 1\bmod 6$, we have
%there is at least
$$\widehat{{N_1}}(3T) \ge \frac{6\cdot T!\cdot T!\cdot T!}{|\mathrm{GL}(\lfloor\log_2 (3T+1) \rfloor, 2)|}.$$
%resolvable STS$(3T)$ of $3$-rank $3T-2$.
\end{proposition}
\begin{proof}
If $T$ satisfies the hypothesis of the proposition,
there is a construction \cite[Theorem~5]{RChWil:kirkman} of resolvable STS
$(\mathrm{GF}(T)\times\{0,1,2\},\bB)$,
$\bB=\bB_0\cup\bB_1\cup\bB_2\cup\tT$,
where $(\mathrm{GF}(T)\times\{i\},\bB_i)$, $i=0,1,2$,
are STS$(T)$ and $\tT$ is a TD$(3,T)$.
It is straightforward that such STS$(3T)$ is orthogonal to $(1,\ldots,1)$ and
$(0,\ldots,0,1, \ldots, 1, 2,\ldots,2)$;
so, its rank is $3T-2$ (it cannot be smaller because $T$ is not divisible by $3$).
Applying $3!\cdot T!\cdot T!\cdot T!$ permutations of points that keep the dual space,
we obtain isomorphic STS, and only $|\mathrm{Aut}(\bB)|$
of permutation give the same result.
So, the number of different STS isomorphic to $\bB$ and dual to $G(3T,1)$ is
$6\cdot T!\cdot T!\cdot T!/|\mathrm{Aut}(\bB)|$.
As proved in \cite{SolTop2000pit}, a Steiner triple system of any order $n$
has no more than $|\mathrm{GL}(\lfloor \log_2 (n + 1)\rfloor,2)|$ automorphisms.
\end{proof}

\begin{example}
 Consider the case $T=7$, $k=2$.
 Proposition~\ref{p:RChWil-constr} gives
 $\widehat{{N_1}}(21)\ge \frac{6\cdot 7!\cdot 7!\cdot 7!}{15\cdot 14\cdot 12\cdot 8} = 38102400 $.
 Considering only the isomorphic TD$(3,7)$ that correspond to the linear $7\times 7$ latin square,
 namely, the isomorphs of
 $$ \tT=\{\{a,b,c\} \,|\,
 a\in\{0,\ldots,6\},\
 b\in\{7,\ldots,13\},\
 c\in\{14,\ldots,20\},\
 a+b+c\equiv 0\bmod 7 \}, $$
 we find $$\widetilde N_3(7)\ge \frac{3!\cdot 7!^3}{|\mathrm{Aut}(\tT)|} \ge \frac {768144384000}{1764}=435456000.$$
 (The exact number $\widehat{{N_1}}(7)$ can be computed from the list of all $7$
 nonequivalent pairs of orthogonal $7\times 7$
 latin squares, see \cite{MK:data}.)
 Substituting to \eqref{eq:b2}, we find the
 following lower bound on the number
 of non-isomorphic STS$(64)$ of $3$-rank $61$:
 $$N(63,2)\ge \frac{38102400^3 \cdot 435456000^{9}} {7!^9 \cdot 9 \cdot 8 \cdot 6} \ge 10^{64}.$$
 % which is much larger than the bound $???$ given in \cite{HCD}.
\end{example}

\section{Resolvable STS of non-minimum $3$-rank}\label{s:nonmin}

According to Propositions~\ref{p:any} and~\ref{p:standard},
we have a representation of every STS$(v)$ whose $3$-rank is not larger than
$v-k-1$.
To evaluate the number of isomorphism classes,
we can use arguments similar to those in Theorem~\ref{th:1}.
However, those arguments work
only if we are sure that the orthogonal space $G(v,k)$
is uniquely determined as the dual space of each of the considered systems;
that is, if the $3$-rank is exactly $v-k-1$. So,
we either need to develop another arguments to estimate the number
of isomorphism classes, or to consider some subclass of the class of all systems
from Proposition~\ref{p:standard} such that every system from
the subclass has $3$-rank exactly $v-k-1$. We choose the second way.
The key fact, in our arguments, is the following lemma.

\begin{lemma}\label{l:key}
Assume that $v=3^k T$, where $k\ge 1$ and $T>3$.
 Under the hypothesis and notation of Proposition~\ref{p:standard},
 there is an STS $(S_0,\bB_0')$ such that the STS
 \begin{equation}\label{eq:3'}
 \bB'=  \bB_{0}' \cup \bigcup_{i=1}^{3^k-1} \bB_{i} \cup \bigcup_{B\in \AG(k)} \tT_B,
\end{equation}
which differs from $\bB $ by only a sub-STS on $S_0$,
has the dual space
$G(v,k)$ and, hence, the rank $v-k-1$.

\end{lemma}

Before proving Lemma~\ref{l:key}, we formulate its direct corollary,
the main result of this section.

\begin{theorem}\label{th:1'}
 Assume that $v=3^k T$,
 where $T\equiv 1,3 \bmod 6$ and $k\ge 1$.
 The number $N(v,k)$ of isomorphism classes of resolvable
 Steiner triple systems on $v$ points
 with $3$-rank exactly $v-k-1$ satisfies
 \begin{equation}\label{eq:b1'}
   N(v,k)
 \geq
 \frac
 {{\widetilde{{N_1}}(T)}^{M-1}\cdot{\widetilde{{N_3}}(T)}^{M(M-1)/6}}
 {T!^M\cdot|AGL(k,3)|}
 \end{equation}
where $M=3^{k}$,
$\widetilde{N_1}(T)$ is the number of resolvable \STS$(T)$,
$\widetilde{N_3}(T)$ is the number of resolvable \TD$(3,T)$
(Latin squares of order $T$ having an orthogonal mate).
\end{theorem}

The proof is the same as for Theorem~\ref{th:1},
with the only difference that
we are not free to choose the subsystem $(S_0,\bB_0)$;
its choice is forced to make
the rank exactly $v-k-1$ in accordance with Lemma~\ref{l:key}.

To prove Lemma~\ref{l:key}, we need two technical results.
The first one generalizes Proposition~\ref{p:any} from STS's to partial STS's of some kind.

% 11111111111111111111111111111111111111111111111111111111111111111111111
% 11111111111111111111111111111111111111111111111111111111111111111111111
% 11111111111111111111111111111111111111111111111111111111111111111111111
% 11111111111111111111111111111111111111111111111111111111111111111111111
% 11111111111111111111111111111111111111111111111111111111111111111111111

\begin{proposition}\label{p:any2}
 Let $(S,\bB)$, $|S|=v$, be an STS orthogonal
 to the vectors $(1,\ldots,1)$ and
 $(0,\ldots,0$, $1,\ldots,1$, $2,\ldots,2)$. If $\bB_0$ be some set of blocks from $\bB$
 restricted by the first $v/3$ points,
 then the dual space of $\bB \setminus \bB_0$ is $G(v,k')$, up to a permutation of points, for some $k'\ge 1$.
\end{proposition}
\begin{proof}
 Let $M$ be the generator matrix of the dual space of $\bB \setminus \bB_0$, so we only need to prove that $M$ is also the generator matrix of $G(v,k').$ Without loss of generality, we can assume that the first column of $M$ is $(1,0,...,0)^T$.

 Claim $(*)$. If $a$ and $b$ are columns of $M$, then $-a-b$ is also a column of $M$. Let $a$ and $b$ be the $j^{th}$ and $k^{th}$ columns of $M$. Firstly, if neither $a$ nor $b$ is in the first $v/3$ positions of $M$, then there is a block $\{j,k,l\}$ in $\bB \setminus \bB_0$ that contains $j$ and $k$. Since all rows of $M$ are orthogonal to the characteristic vector of this block, the $l^{th}$ column $c$ satisfies $a+b+c=0$, i.e., $c=-a-b$. This proves $(*)$.  Secondly, $a$ is in the first $v/3$ positions but $b$ is not. Because the second element of $a$ is 0, and the second element of $b$ is 1 or 2, so the the second element of $-a-b$ is 2 or 1 over $GF(3)$, it means that $-a-b$ cannot be in the first $v/3$ positions of $M$. Because two of the three columns $a$, $b$ and $-a-b$ are not in the first $v/3$ positions, so if $-a-b$ is in the $l^{th}$ position, $\{j,k,l\}$ will be a block of $\bB \setminus \bB_0$, therefore, $-a-b$ is a column of $M$. Finally, both $a$ and $b$ are in the first $v/3$ positions, then let $c$ be a column of $M$ that does not lie in the first $v/3$ positions. By the second case, $-a-c$ and $-b-c$ are columns of $M$, and they are not in the first $v/3$ positions of $M$. So $-(-a-c)-(-b-c)=a+b-c$ is a column of $M$, we know its not in the first $v/3$ positions because its second element is not 0. Therefore, $-(a+b-c)-c=-a-b$ is a column of $M$. These prove $(*)$.

 Claim $(**)$. If $c$ and $d$ are columns of $M$, then $c+d-(1,0,...,0)^T$ is also a column of $M$. This is proved by applying $(*)$ with $a=c$, $b=d$ first, and then with $a=-c-d$, $b=(1,0,...,0)^T$.

 The last claim means that the set of columns of the matrix $M^\prime$ obtained from $M$ by removing the first row is closed under addition. Since there are $k^\prime$ linearly independent columns, this set contains all possible $3^{k^\prime}$ columns of height $k^\prime$.

 It remains to prove that different columns $a$ and $b$ occur the same number of times in $M$. Let $J$ and $K$ be the sets of positions in which $M$ has the columns $a$ and $b$, respectively. And let $l$ be a position of the column $-a-b$.

 Case $(1)$, if all $a$s and $b$s are not in the first $v/3$ positions of $M$, then for each $j$ from $J$, there is $k$ from $K$ such that $\{j,k,l\}$ is a block of $\bB \setminus \bB_0$. Moreover, different $j$s correspond to different $k$s. This shows that $|J|\leq|K|$. Similarly, $|K|\leq|J|$.

 Case $(2)$, if $a$s are in the first $v/3$ positions but $b$s are not. By the proof of claim $(*)$, $-a-b$ can't be in the first $v/3$ positions of $M$, so for each $j$ from $J$, there is still $k$ from $K$ such that $\{j,k,l\}$ is a block of $\bB \setminus \bB_0$. Different $j$s correspond to different $k$s, it shows that $|J|\leq|K|$. Similarly, $|K|\leq|J|$.

 Case $(3)$, if all $a$s and $b$s are in the first $v/3$ positions, then let $c$ be a column of $M$ that does not lie in the first $v/3$ positions. Let $A$ be the sets of positions in which $M$ has the column $c$. By case $(2)$, we know $|J|=|A|$ and $|K|=|A|$, so $|J|=|K|$.
\end{proof}

% 22222222222222222222222222222222222222222222222222222222222222222222222
% 22222222222222222222222222222222222222222222222222222222222222222222222
% 22222222222222222222222222222222222222222222222222222222222222222222222
% 22222222222222222222222222222222222222222222222222222222222222222222222
% 22222222222222222222222222222222222222222222222222222222222222222222222

\begin{proposition}\label{p:perm}
For every $t\ge 0$ and $T>3$ divisible by $3^t$,
there is a permutation $\pi$ of the $T$ coordinates
such that the spaces $G(T,t)$ and $\pi(G(T,t))$
intersect in the one-dimensional
subspace spanned by the all-one vector.
\end{proposition}

\begin{proof}
Since $\dim(G(T,t))=\dim(\pi(G(T,t)) = t+1$,
the equation $\dim(G(T,t) \cap \pi(G(T,t)) = 1$ is equivalent to $\dim(G(T,t) + \pi(G(T,t))) = 2t+1$.

The case $t=0$ is trivial.

In the case $t=1$ (hence $T\ge 6$),
we can choose $\pi$ such that the generator matrices of $G(T,t)$ and $\pi(G(T,t))$ are of forms
$$
\left(
\begin{array}{cccccccccc}
 1&1&...&1&1&...&1&1&...&1 \\
 0&0&...&0&1&...&1&2&...&2 \\
\end{array}
\right),
\qquad
\left(
\begin{array}{cccccccccc}
 1&1&... \\
 0&1&... \\
\end{array}
\right),
$$
respectively, and we see that the second row of the second matrix is not in $G(T,t)$.

 Assume $t\ge 2$. Denote by $\mathrm{Sym}_T$ the set of all permutations of $T$ coordinates.
 For any $\pi\in \mathrm{Sym}_T$,
 let $M$ and $\pi(M)$ be
 generator matrices of $G(T,t)$ and $\pi(G(T,t))$, respectively,
with the first row being the all-one vector.
 So, the problem is equivalent
 to finding a permutation $\pi\in \mathrm{Sym}_T$
 such that $\dim(G(T,t) + \pi(G(T,t))) = 2t+1$,
 i.e., the rank of the ${2(t+1)\times T}$ matrix $P=\left( \begin{array}{c}
M\\
\pi(M)\\
\end{array}\right )$ is $2t+1$.
Let $P_\pi^\prime=\left( \begin{array}{c}
M^\prime\\
\pi(M^\prime)\\
\end{array}\right )$, where the matrices $M^{\prime}$ and $\pi(M^{\prime})$ 
are obtained from $M$ and $\pi(M)$, respectively, by removing the first row.

We now state that the permutation $\pi$ can be chosen in such a way that $P'_\pi$
has a ${2t\times(2t+1)}$ submatrix
$$Q=\left( \begin{array}{ccc}
%1&\bar{1}&\bar{1}\\
\bar{0}^T&I_t & 2I_t\\
%1&\bar{1}&\bar{1}\\
\bar{0}^T&C& I_t\\
\end{array}\right ), \qquad \mbox{where }C=\left( \begin{array}{cc}
\bar 0&2\\
 2I_{t-1} &  \displaystyle {{\bar 0}^T \atop 1}\\
\end{array}\right ),$$
$I_j$ is the identity matrix of size $j\times j$, 
and  $\bar{0}$ is the all-zero vector.
Indeed, $M$ contains all possible columns of height $t$. 
As the $2t+1$ columns of the matrix $(\bar{0}^T\  I_t \ 2I_t)$ are different, 
it is a submatrix of $M$. 
Similarly, $(\bar{0}^T\ C\  I_t)$ is a submatrix of $M$,
and we can choose $\pi$ in such a way that $Q$ is a submatrix of $P'_\pi$.

With $\pi$ defined as above, $P_\pi$ contains a $(2t+1)\times (2t+1)$ submatrix
$\left( \begin{array}{ccc}
1&\bar{1}&\bar{1}\\
\bar{0}^T&I_t & 2I_t\\
\bar{0}^T&C& I_t\\
           \end{array}\right )$
of non-zero determinant
$      
           \det\left( \begin{array}{cc}
           I_t&2I_t\\
           C&I_t\\
           \end{array}\right )=\det\left( \begin{array}{cc}
                                  I_t&2I_t\\
                                  0&I_t-2C\\
                                  \end{array}\right )\ne 0.$
So $P_\pi$ has rank $2t+1$ and the spaces $G(T,t)$ and $\pi(G(T,t))$ intersect in a one-dimensional subspace.
\end{proof}

Now, we are ready to proof Lemma~\ref{l:key}.

\begin{proof}[of Lemma~\ref{l:key}]
Assume that $v=3^k T$, $k\ge 1$, $T>3$, and $(S,\bB)$ is an STS of order $v$ and rank at most $v-k-1$.
By Proposition~\ref{p:any}, we can assume without loss of generality
that $\bB$ is orthogonal to $G(v,k)$.
By Proposition~\ref{p:standard}  and under its notation,
$\bB$ has the following decomposition:
\begin{equation}\label{eq:1''}
 \bB = \bB_{0}  \cup \bB^-, \qquad \mbox{where } \bB^- = \bigcup_{i=1}^{3^k-1} \bB_{i} \cup \bigcup_{B\in AG(k)} \tT_B
\end{equation}
% Since $\bB_0$ is an STS of order $T>3$, we observe that $T>6$,
% which will be used when utilyzing Proposition~\ref{p:perm} in the arguments below.
By Proposition~\ref{p:any2}, the dual space of $\bB^-$ is equivalent to $G(v,k')$ for some {$k'\ge k$}
(it should include the dual space $G(v,k)$ of $\bB$),
and without loss of generality assume that it is precisely $G(v,k')$.

Now consider the STS $\bB_0$ of order $T$.
Let its rank be $T-l-1$ for some $l$.
Then its dual space is equivalent to $G(T,l)$,
and we consider a permutation $\sigma$ of coordinates $0$, \ldots, $T-1$ such that
the dual space of $\sigma(\bB_0)$ is precisely $G(T,l)$.
We now define $t:=\max(l,k'-k)$ and take a permutation $\pi$ from Proposition~\ref{p:perm},
such that $$\dim(G(T,t) \cap \pi(G(T,t)))=1.$$

We state that the conclusion of the lemma is satisfied 
with $\bB_0'= \pi(\sigma(\bB_0))$.
That is, the rank of the system $\bB'$ obtained from $\bB$ 
with replacing $\bB_0$ by $\bB_0'$
is exactly $v-k-1$. 
To see this, 
we assume for the contrary that the dual space of $\bB'$
has a vector $x$ in $G(v,k') \setminus G(v,k)$.
Then the values in the first $T$ coordinates
of $x$ are not constant,
and they form a vector from $G(T,k'-k)$ 
(and hence, from $G(T,t)$), say $y$.
By Proposition~\ref{p:perm}, 
it does not belong to $\pi(G(T,t))$,
and hence it does not belong to $\pi(G(T,l))$,
i.e., it is not orthogonal to $\pi(\sigma(\bB_0))=\bB_0'$.
Therefore $x$ is not orthogonal to $\bB'$, a contradiction.
So, by rotating only one sub-STS of order $T$, we excluded all orthogonal vectors that are not in $G(v,k)$.
\end{proof}

\section{Conclusion}

In this paper, we have proven a hyperexponential (in $k$) lower bound
for the number of isomorphism classes of resolvable STS$(3^k T)$
of $3$-rank $3^k T - k -1$, for most (but not all) values of odd $T>3$. We separately solved the cases
$T\equiv 1\bmod 6$ (Theorem~\ref{th:2}), $T\equiv 15\bmod 18$ (Theorem~\ref{th:1}), and $T\equiv 3,9\bmod 18$ (Theorem~\ref{th:1'}).
The case $T=3$ was previously considered in \cite{JMTW:STS27}; the case $T=1$ corresponds to the affine geometry,
which is a unique STS$(3^k)$ of $3$-rank $3^k-k-1$.
Theorems~\ref{th:1} and~\ref{th:1'} cover all corresponding cases, but Theorem~\ref{th:2} is conditional:
to produce a nontrivial lower bound, it needs at least one resolvable STS$(3T)$ of non-maximum $3$-rank.
If $T$ is a prime power, such STS were constructed in \cite[Theorem~5]{RChWil:kirkman}.
The existence of resolvable STS$(3T)$ of $3$-rank $3T-2$ for the other values of $T$, $T\equiv 1\bmod 6$,
remains an actual research problem.

It would be quite interesting, but expectedly more difficult to consider similar
asymptotics with respect to the $2$-rank. Another interesting problem is the evaluation
of the number of doubly-resolvable STS of limited $3$- or $2$-rank
(an STS is doubly resolvable if there are two resolutions such that two parallel classes from different
resolutions have no more than one block in common).

% \bibliographystyle{plain}
% \bibliography{../../k}
% \end{document}

\providecommand\href[2]{#2} \providecommand\url[1]{\href{#1}{#1}}
  \def\DOI#1{{\small {DOI}:
  \href{http://dx.doi.org/#1}{#1}}}\def\DOIURL#1#2{{\small{DOI}:
  \href{http://dx.doi.org/#2}{#1}}}

\end{document}